\documentclass[10pt]{amsart}
\usepackage[british,UKenglish,USenglish,english,american]{babel}
\usepackage[utf8]{inputenc}
\usepackage{braket}
\usepackage{xfrac}
\usepackage{amsfonts, amssymb, amsmath, amsthm, amscd}
\usepackage{mathtools}
\usepackage{tikz}
\usepackage{empheq}
\usepackage[left=4cm,right=4cm,top=4cm,bottom=4cm]{geometry}
\usepackage{color}
\usepackage{epsfig}  	
\usepackage{epic,eepic}       
\def\bysame{\leavevmode\hbox to3em{\hrulefill}\thinspace}
\DeclarePairedDelimiterX{\norm}[1]{\lVert}{\rVert}{#1}
\newtheorem{thm}{Theorem}[section]
\newtheorem{prop}[thm]{Proposition}
\newtheorem{lem}[thm]{Lemma}

\theoremstyle{definition}
\newtheorem{definition}[thm]{Definition}
\newtheorem{example}[thm]{Example}

\theoremstyle{remark}
\newtheorem{remark}[thm]{Remark}

\begin{document}

\title[The Gauss map on translational Riemannian manifolds]{The Gauss map on translational Riemannian manifolds and the topology of hypersurfaces}

\author{Eduardo R. Longa}
\address{Instituto de Matem\'{a}tica e Estat\'{i}stica, Universidade Federal do Rio Grande do Sul, Porto Alegre, Brasil} 
\email{eduardo.longa@ufrgs.br}

\author{Jaime B. Ripoll}
\address{Instituto de Matem\'{a}tica e Estat\'{i}stica, Universidade Federal do Rio Grande do Sul, Porto Alegre, Brasil} 
\email{jaime.ripoll@ufrgs.br}

\begin{abstract}
We introduce the notion of translational Riemannian manifolds and define a Gauss map for orientable immersed hypersurfaces lying in these ambients,
an associated translational curvature and prove a Gauss-Bonnet theorem. We also use this Gauss map to prove that if $M^{n}$ is a compact, connected
and oriented immersed hypersurface of the unit sphere $\mathbb{S}^{n+1}$ ($n\geq2$) contained in a geodesic ball of radius $R$ and whose principal
curvatures are strictly bigger than $\tan\left( R/2 \right)$, then $M$ is diffeomorphic to $\mathbb{S}^{n}$. Additionally, we show that for any 
$\varepsilon\in(0,\sqrt{2}-1)$ there exists a compact, connected and oriented immersed hypersurface $M_{\varepsilon}$ of $\mathbb{S}^{n+1}$ whose
principal curvatures are strictly bigger than $\varepsilon \tan \left( R/2 \right)$ but $M_{\varepsilon}$ is not homeomorphic to a sphere. Finally, 
using this previous result, we reobtain a theorem of Qiaoling Wang and Changyu Xia (\cite{xia06}) which asserts that if a compact and oriented
hypersurface of $\mathbb{S}^{n+1}$ is contained in an open hemisphere and has nowhere zero Gauss-Kronecker curvature, then it is diffeomorphic to 
$\mathbb{S}^n$.
\end{abstract}

\maketitle
\section{Introduction}

In this paper we define a Gauss map for an orientable hypersurface of a Riemannian manifold by extending the notion of translation of the Euclidean
space --- used to define the Euclidean Gauss map --- to more general ambient spaces. We then use this map to obtain results on the topology of the 
hypersurface.

The main objects we are interested in are \emph{translational manifolds}, defined as follows:

\begin{definition} A translational Riemannian manifold is a pair $\left( \overline{M}, \Gamma \right)$, where $\overline{M}$ is an $(n+1)$-dimensional
Riemannian manifold, $\Gamma : T \overline{M} \to \overline{M} \times V$ is a smooth \emph{vector bundle} map, and $V$ is an $(n+1)$-dimensional real
vector space with an inner product such that the map $\Gamma_p : T_p \overline{M} \to V$ implicitly defined by

\begin{align*}
v \mapsto \Gamma(p,v) = \left( p,\Gamma_{p}(v) \right)
\end{align*} 

\noindent is an isometry for every point $p \in \overline{M}$. The manifold $\overline{M}$ is said to be equipped with a translational structure.
\end{definition}

The maps $\Gamma_{p}$ are to be thought as translations, as means of identifying the tangent spaces to $\overline{M}$ with the vector space $V$. Notice that any
translational Riemannian manifold is parallelisable, that is, it admits as many linearly independent vector fields as its dimension. Conversely, any paralellisable 
manifold has infinitely many translational structures (see next section). Natural examples of such manifolds are obtained by considering left translation on Lie groups
with a left invariant metric and parallel transport to a fixed point, defined on the manifold minus the cut locus of the point (Examples \ref{Left translation} and 
\ref{Parallel transport}). The former case was studied in \cite{ripoll91}, and in the present paper we investigate the latter when $\overline{M}$ is an Euclidean
sphere. Both scenarios reduce to the usual Euclidean translations when $\overline{M} = \mathbb{R}^{n+1}$.

On a translational Riemannian manifold $\left( \overline{M}, \Gamma \right)$ one can naturally define a \emph{Gauss map} for an orientable
hypersurface $M^{n}$ of $\overline{M}.$ Indeed, being $\eta$ a unit normal vector field along $M$, the Gauss map $\gamma:M\rightarrow\mathbb{S}^{n}$ 
of $M$ is defined by $\gamma(p) = \Gamma_{p}(\eta(p))$, where $\mathbb{S}^n \subset V$ is the unit sphere centered at the origin of $V$. It is easy to
see that $\Gamma_{p}^{-1}\circ D \gamma(p)$ is a linear map on $T_{p} M$ and it is shown (Proposition \ref{relationship}) that $\Gamma_{p}^{-1}\circ D \gamma(p) 
= - (A_{p}+ \alpha_{p})$, where $A_{p}$ is the shape operator of $M$ and $\alpha_{p}$ is a \emph{translational shape operator} that depends essentially on $\Gamma$.
The \emph{translational curvature} $\kappa_{\Gamma}$ of $M$ is defined by $\kappa_{\Gamma}(p) = \det \left( \Gamma_{p}^{-1}\circ D \gamma(p) \right)$. Observe that
$\kappa_{\Gamma}$ is the Gauss-Kronecker curvature of $M$ when $M$ is a hypersurface of $\mathbb{R}^{n+1}$ and $\gamma$ is the usual Gauss map associated to the
translation $\Gamma_{p}(v) = v$.

Our first result is a Gauss-Bonnet theorem:

\begin{thm} \label{gaussbonnet} Let $\left( \overline{M}, \Gamma \right)$ be a translational Riemannian manifold and $M^n$ a compact, connected and orientable immersed
hypersurface of even dimension of $\overline{M}$, and denote by $\omega$ the volume element of $M$ induced by the metric of $\overline{M}$. Then

\begin{align*}
\int_M \kappa_\Gamma \, \omega = \frac{c_n}{2} \chi(M),
\end{align*}

\noindent where $c_n$ is the volume of $\mathbb{S}^n \subset V$ and $\chi(M)$ is the Euler characteristic of $M$.
\end{thm}

If $M$ is compact and its translational curvature is everywhere nonzero, then $\gamma : M \to \mathbb{S}^n$ is a local diffeomorphism, and hence 
a global diffeomorphism, since $\mathbb{S}^n$ is simply connected (for $n \geq 2$). An instance which follows from this remark and from the above mentioned formula 
$D \gamma(p) = -\Gamma_{p} \circ (A_{p}+ \alpha_{p})$ is that if $M$ is a compact, connected and orientable hypersurface of a Lie group with a bi-invariant metric and its
principal curvatures have the same sign, then $\gamma$ is a diffeomorphism (see Theorem 9 of \cite{ripoll91}). Also using this remark, we prove:

\begin{thm} \label{our rigidity} Let $M^n$ be a compact, connected and oriented immersed hypersurface of $\mathbb{S}^{n+1}$, $n \geq 2$, and let $R$ be the radius of the
smallest geodesic ball containing $M$. If the principal curvatures $\lambda_1, \dots, \lambda_n$ of $M$ satisfy

\begin{align*}
|\lambda_i(p)| > \tan \left( \frac{R}{2}\right), \quad \forall \, p \in M, \; \forall \, i \in \{1, \dots, n \},
\end{align*}

\noindent then $M$ is diffeomorphic to $\mathbb S^n$. Moreover, for any $\varepsilon \in (0, \sqrt{2} - 1)$ there exists a compact, connected and oriented immersed hypersurface 
$M_\varepsilon$ of $\mathbb{S}^{n+1}$ whose principal curvatures satisfy

\begin{align}
|\lambda_i(p)| > \varepsilon \tan \left( \frac{R}{2}\right), \quad \forall \, p \in M_\varepsilon, \; \forall \, i \in \{1, \dots, n \}, \label{epsilon_inequality}
\end{align}

\noindent but $M_\varepsilon$ is not homeomorphic to a sphere.
\end{thm} 

Lastly, we provide an alternative proof for Theorem 1.1 of \cite{xia06}:

\begin{thm} \label{xia rigidity} Let $M^n$ be a compact, connected and oriented immersed hypersurface of $\mathbb{S}^{n+1}$, $n \geq 2$, with non-vanishing Gauss-Kronecker
curvature. If $M$ is contained in an open hemisphere, then $M$ is diffeomorphic to $\mathbb S^n$.
\end{thm}

\section{Gauss map} \label{gauss}

Let $\overline{M}^{n+1}$ be a Riemannian manifold. As stated in the Introduction, $\overline{M}$ admits a translational structure $\Gamma$ if and only if it is parallelisable, 
for if $\Gamma$ is given, choose a basis $\{ v_{1},..,v_{n+1} \}$ of $V$ and define $V_{i}(p)= \Gamma_{p}^{-1}(v_{i})$. Conversely, if $\overline{M}$ is parallelisable we may 
consider, after an orthonormalisation process, vector fields $V_{i}$ such that $\left\langle V_{i},V_{j} \right\rangle = \delta_{ij}$, $1 \leq i,j \leq n+1$. Then, choose a point
$p_{0}$ in $\overline{M}$, set $V = T_{p_{0}} \overline{M}$ and define

\begin{align*}
\Gamma_{p}(v) = \sum_{i=1}^{n+1} \left\langle v, V_{i}(p) \right\rangle V_{i}(p_{0}), \quad p \in \overline{M}, \, v \in T_{p}\overline{M}.
\end{align*}    

Let $\left( \overline{M}, \Gamma \right)$ be a translational Riemannian manifold and $f : M^n \to \overline{M}$ an immersion of an orientable manifold $M$ into $\overline{M}$. 
The following constructions are purely local, so we identify small neighbourhoods of $M$ with their images via $f$, and the tangent spaces to $M$ with their images via $Df$.
Let $\eta : M \to T \overline{M}$ be a unit normal vector field along $f$, and let $\mathbb{S}^n$ be the unit sphere of $V$. 

\begin{definition} The Gauss map $\gamma : M \to \mathbb{S}^n$ associated to the normal vector field $\eta$ is given by

\begin{align*}
\gamma(p) = \Gamma_p(\eta(p)), \quad p \in M.
\end{align*}
\end{definition}

The tangent space of $V$ at any point is canonically isomorphic to $V$, and via this isomorphism the tangent space of $\mathbb{S}^n$ at a point $x$ is just $\{x\}^{\perp}$.
Thus, the derivative $D \gamma(p)$ maps $T_p M$ into $T_{\gamma(p)} \mathbb{S}^n = \{\gamma(p)\}^{\perp}$ and $\Gamma_p^{-1}$ maps the latter back into $T_p M$. This makes 
possible the following: 

\begin{definition} The $\Gamma$-curvature of $M$ is the map $\kappa_{\Gamma} : M \to \mathbb{R}$ given by

\begin{align*}
\kappa_{\Gamma}(p) = \det \left( \Gamma_p^{-1} \circ D \gamma(p) \right), \quad p \in M.
\end{align*}
\end{definition}

Next, we define a special type of vector field that will play an important role.

\begin{definition} Given a vector $X \in T_p \overline{M}$, the vector field $\widetilde{X} \in \mathfrak{X}(\overline{M})$ defined by

\begin{align*}
\widetilde{X}(q) = \left( \Gamma_q^{-1} \circ \Gamma_p \right)(X), \quad q \in \overline{M}
\end{align*}

\noindent is called the $\Gamma$-invariant (or simply invariant) vector field of $\overline{M}$ associated with $X$.
\end{definition}

\begin{example}[The Euclidean translation] If $\overline{M} = \mathbb{R}^{n+1}$ and $\Gamma : T \mathbb{R}^{n+1} \to \mathbb{R}^{n+1} \times \mathbb{R}^{n+1}$ is the identity,
then the Gauss map $\gamma$ for an orientable hypersurface $M$ is the ordinary one. The invariant vector fields of $\mathbb{R}^{n+1}$ are the constant vector fields and
$\kappa_\Gamma$ is the Gauss-Kronecker curvature of $M$.
\end{example}

\begin{example}[Left translation on Lie groups] \label{Left translation} More generally, let $\overline{M} = G$ be a Lie group and $V = \mathfrak{g}$ be the Lie algebra of $G$,
considered as the tangent space of $G$ at the identity. Choose a left invariant metric for $G$ and define $\Gamma : TG \to G \times \mathfrak{g}$ by

\begin{align*}
\Gamma(g, v) = \left( g, DL_{g^-1}(g) \cdot v \right), \quad (g, v) \in TG,
\end{align*}

\noindent where $L_x : y \mapsto xy$ is the left translation. Here, the $\Gamma$-invariant vector fields are the left invariant vector fields of $G$. This is the setting studied in \cite{ripoll91}. 
\end{example}

\begin{example}[Parallel transport] \label{Parallel transport} Assume $\overline{M}$ is a Cartan-Hada\-mard manifold, that is, a complete, connected and simply connected Riemannian
manifold with nonpositive sectional curvature. Given a point $p_0\in \overline{M}$, the exponential map at $p_0$ is, by Hadamard's Theorem, a diffeomorphism from $T_{p_0} \overline{M}$ 
onto $\overline{M}$, so that every point $p$ can be joined to $p_0$ by a unique geodesic. Setting $V = T_{p_0} \overline{M}$, we may then define $\Gamma_p : T_p \overline{M} \to V$
by choosing $\Gamma_p(v)$ as being the parallel transport of $v \in T_p \overline{M}$ to $ T_{p_0}\overline{M}$ along this geodesic. Thus, the invariant vector fields here are the
parallel vector fields along the geodesic rays issuing from $p_0$.

More generally, given any complete Riemannian manifold $\overline{M}$ and a point $p_0$ in $\overline{M}$, we can define the parallel transport to $T_{p_0} \overline{M} $ on 
$\overline{M} \setminus C_{p_0}$ as above, where $C_{p_0}$ is the cut locus of $p_0$ (\cite{SY}, Chapter I). We study this case in detail on the sphere (Section \ref{sphere}).
\end{example}

We next describe the geometry of the Gauss map. Let $\overline{\nabla}$ be the Levi-Civita connection of $\overline{M}$. Recall that the shape operator of $M$ is the section
$A$ of the vector bundle $\operatorname{End}(TM)$ of endomorphisms of $TM$ given by

\begin{align*}
A_p(X) = - \overline{\nabla}_X \eta, \quad p \in M, \, X \in T_p M.
\end{align*}

\noindent Similarly, we define another section of $\operatorname{End}(TM)$, which depends additionally on the choice of the translation $\Gamma$.

\begin{definition} The invariant shape operator of $M$ is the section $\alpha$ of the bundle $\operatorname{End}(TM)$ given by

\begin{align*}
\alpha_p(X) = \overline{\nabla}_X \widetilde{\eta(p)}, \quad p \in M, \, X \in T_p M.
\end{align*}
\end{definition}

The proposition below establishes a relationship between $\gamma$ and the extrinsic geometry of $M$.

\begin{prop} \label{relationship} For any $p \in M$, the following identity holds:

\begin{align*}
\Gamma_p^{-1} \circ D \gamma(p) = - \left( A_p + \alpha_p \right).
\end{align*}
\end{prop}

\begin{proof}
Fix $p \in M$ and an orthonormal basis $\{X_1, \dots, X_{n+1} \}$ of $T_p \overline{M}$ such that $X_{n+1} = \eta(p)$. The vector fields $\widetilde{X}_1, \dots, \widetilde{X}_{n+1}$ 
form a global orthonormal referential of $T \overline{M}$, so that we can write 

\begin{align} 
\eta = \sum_{i=1}^{n+1} a_i \widetilde{X}_i \label{normal}
\end{align}

\noindent for certain functions $a_i \in C^{\infty}(M)$. Notice that $a_i(p) = 0$ for $i \in \{1, \dots, n \}$ and $a_{n+1}(p) = 1$.

For $y \in M$ we have

\begin{align*}
\gamma(y) = \Gamma_y(\eta(y)) = \Gamma_y \left(\sum_{i=1}^{n+1} a_i(y) \widetilde{X}_i (y) \right) = \sum_{i=1}^{n+1} a_i(y) \Gamma_p(X_i).
\end{align*}

\noindent Therefore, if $X \in T_p M$,

\begin{align} 
\Gamma_p^{-1}(D \gamma(p) \cdot X) = \Gamma_p^{-1} \left( \sum_{i=1}^{n+1} X(a_i) \Gamma_p(X_i) \right) = \sum_{i=1}^{n+1} X(a_i) X_i. \label{composta} 
\end{align} 

\noindent From (\ref{normal}) and (\ref{composta}) we obtain

\begin{align*}
-A_p(X) &= \overline{\nabla}_X \eta
         = \sum_{i=1}^{n+1} \overline{\nabla}_X(a_i \widetilde{X}_i)
         = \sum_{i=1}^{n+1} \left[ a_i(p) \overline{\nabla}_X \widetilde{X}_i + X(a_i) \widetilde{X}_i (p) \right]  \\ 
        &= \overline{\nabla}_X \widetilde{X}_{n+1} + \sum_{i=1}^{n+1} X(a_i) X_i = \alpha_p(X) + \Gamma_p^{-1}(D \gamma(p) \cdot X),
\end{align*}

\noindent which gives the desired result.
\end{proof}

We now provide the proof for our Gauss-Bonnet theorem.

\begin{proof}[Proof of Theorem \ref{gaussbonnet}]
Let $\sigma$ be the volume form of $\mathbb{S}^n$ induced by the metric on $V$. From the fact that $\Gamma$ restricts to isometries in each fibre and from the definition of
$\kappa_\Gamma$, it follows that $\gamma^{*} \sigma = \kappa_\Gamma \, \omega$. Then, the change of variables formula yields 

\begin{align*}
\int_M \kappa_\Gamma \, \omega = \int_M \gamma^{*} \sigma = \operatorname{deg}(\gamma) \int_{\mathbb{S}^n} \sigma = c_n \operatorname{deg}(\gamma).
\end{align*}

\noindent It remains to show that $\operatorname{deg}(\gamma) = \frac{1}{2} \chi(M)$. For this, define $\widetilde{\Gamma} : TM \to T \mathbb{S}^n$ by

\begin{align*}
\widetilde{\Gamma}(p, v) = \left( \gamma(p), \Gamma_p(v) \right), \quad (p, v) \in TM.
\end{align*}

\noindent We have the following vector bundle map

\begin{align*}
\begin{CD}
TM @>\tilde{\Gamma}>> T \mathbb{S}^n\\
@VVV @VVV\\
M @>\gamma>> \mathbb{S}^n
\end{CD}
\end{align*}

\noindent Therefore, if $e(M)$ and $e(\mathbb{S}^n)$ are the Euler classes of $M$ and $\mathbb{S}^n$, and since $n$ is even, we obtain
\begin{align*}
\chi(M) &= \left( e(M), [M] \right) = \left( \gamma^{*}(e(\mathbb{S}^n)), [M] \right) \\
        &= \left( e(\mathbb{S}^n), \gamma^{*}([M]) \right) = \left( e(\mathbb{S}^n), \operatorname{deg}(\gamma)[\mathbb{S}^n] \right) \\
        &= \operatorname{deg}(\gamma) \left( e(\mathbb{S}^n), [\mathbb{S}^n] \right) = \operatorname{deg}(\gamma) \chi(\mathbb{S}^n) = 2 \operatorname{deg}(\gamma),
\end{align*}
where $[\,\cdot\,]$ indicates fundamental class in homology and $(\cdot, \cdot)$ the duality between homology and cohomology.
\end{proof}
 
\section{Topological rigidity of hypersurfaces of the sphere} \label{sphere}

In this section we will investigate the earlier constructions in the following situation. Let $\overline{M}$ be the unit sphere $\mathbb{S}^{n+1} \subset \mathbb{R}^{n+2}$ with 
a point $-p_0$ deleted, which we will denote by $\mathbb{S}_{-p_0}^{n+1}$, and let $V$ be the tangent space of the sphere at $p_0$. The metrics of $\mathbb{S}^{n+1}$ and $V$ are 
those induced from $\mathbb{R}^{n+2}$. Given two non-antipodal points $p, q$ in the sphere, let $\tau_p^q : T_p\mathbb{ S}^{n+1} \to T_q \mathbb{ S}^{n+1}$ be the parallel 
transport along the unique geodesic joining $p$ to $q$ (we agree that $\tau_p^p$ is the identity of $T_p \mathbb{S}^{n+1}$). Since this map is a linear isometry, we define 
$\Gamma : T \mathbb{S}_{-p_0}^{n+1} \to \mathbb{S}_{-p_0}^{n+1} \times V$ by

\begin{align*}
\Gamma(p, v) = \left( p, \tau_p^{p_0}(v) \right), \quad (p, v) \in  T \mathbb{S}_{-p_0}^{n+1}.
\end{align*}

If $M^n$ is an orientable immersed hypersurface of $\mathbb{S}^{n+1}$ not containing $-p_0$ and $\eta : M \to \mathbb{R}^{n+2}$ is a unit normal vector field (tangent to the sphere),
the Gauss map $\gamma : M \to \mathbb{S}^n$ calculated at a point $p$ is just the parallel transport of the normal $\eta(p)$ to $M$ along the geodesic joining $p$ to $p_0$. The next 
proposition contains the relevant information we will need.

\begin{prop} Let $p$ and $q$ be non-antipodal points in $\mathbb{S}^{n+1}$. With the above notations, the following formulae hold:

\begin{itemize}
\item[(i)] 
\begin{align*}
\tau_{p}^q(v) = - \left[ \frac{\langle v, q \rangle}{1 + \langle q, p \rangle} \right] (q + p) + v, \quad v \in T_p \mathbb{S}^{n+1}. 
\end{align*}
\item[(ii)]
\begin{align*}
\gamma(p) = - \left[ \frac{\langle \eta(p), p_0 \rangle}{1 + \langle p, p_0 \rangle} \right] (p + p_0) + \eta(p). 
\end{align*}
\item[(iii)]
\begin{align*}
\alpha_p(X) = \left[ \frac{\langle \eta(p), p_0 \rangle}{1 + \langle p, p_0 \rangle} \right] X, \quad X \in T_p M.
\end{align*}
\end{itemize}
\end{prop}

\begin{proof}
To prove the first two items, let $\beta : [0, t_q] \to \mathbb{S}^{n+1}$ be the unit speed geodesic joining $p$ to $q$, given by

\begin{align*}
\beta(t) = (\cos t) p + (\sin t) \overline{q}, \quad t \in [0, t_q],
\end{align*}

\noindent where 

\begin{align*}
\overline{q} = \frac{q - \langle q, p \rangle p}{\norm{q - \langle q, p \rangle p}} = \frac{q - \langle q, p \rangle p}{\sqrt{1 - \langle q, p \rangle^2}}.
\end{align*}

For fixed $v \in T_p \mathbb{S}^{n+1}$, let $X : [0, t_q] \to \mathbb{R}^{n+2}$ be the parallel vector field along $\beta$ and tangent to the sphere with prescribed 
initial value $X(0) = v$. Differentiating $\langle X, \beta \rangle \equiv 0$, we obtain

\begin{align}
- \langle X', \beta \rangle \equiv \langle X, \beta' \rangle, \label{X}
\end{align}

\noindent and since $X$ and $\beta'$ are parallel along $\beta$, $\langle X, \beta' \rangle$ is constant, equal to $C \in \mathbb{R}$, say, with

\begin{align*}
C = \langle X(0), \beta'(0) \rangle = \langle v, \overline{q} \rangle = \frac{\langle v, q \rangle}{\sqrt{1- \langle q, p \rangle^2}}. 
\end{align*}

The equation for $X$ to be a parallel vector field is $X' - \langle X', \beta \rangle \beta \equiv 0$. Writing $X = (x_1, \dots, x_{n+2})$, using (\ref{X}) and 
the expression for $\beta$, we have

\begin{align*} 
X'(t) = -C \left[ (\cos t) p + (\sin t) \overline{q} \right], \quad t \in [0, t_q].
\end{align*}

\noindent The solution satisfying $X(0) = v$ is then 

\begin{align*}
X(t) = C \left[ (\cos t - 1) \overline{q} - (\sin t) p  \right] + v, \quad t \in [0, t_q].
\end{align*}

\noindent Noticing that $\cos t_q = \langle q, p \rangle$ and $\sin t_q = \sqrt{1- \langle q, p \rangle^2}$, we finally obtain

\begin{align*}
\tau_{p}^q(v) = X(t_p) = - \left[ \frac{\langle v, q \rangle}{1 + \langle q, p \rangle} \right] (q + p) + v
\end{align*}

\noindent and

\begin{align*}
\gamma(p) = \tau_p^{p_0}(\eta(p)) = - \left[ \frac{\langle \eta(p), p_0 \rangle}{1 + \langle p, p_0 \rangle} \right] (p + p_0) + \eta(p), 
\end{align*}

\noindent as required.

For the last item, let $v = \gamma(p)$. Recall that $\widetilde{v} \in \mathfrak{X}\left( \mathbb{S}_{-p_0}^{n+1} \right)$ is the invariant vector field associated with $v$, and 
$\widetilde{v} = \widetilde{\eta(p)}$. From (i) we have

\begin{align}
\widetilde{v}(q) = \tau_{p_0}^q(v) =  - \left[ \frac{\langle v, q \rangle}{1 + \langle q, p_0 \rangle} \right] (q + p_0) + v, \quad q \in \mathbb{S}_{-p_0}^{n+1}. \label{v}
\end{align}

\noindent If $\widetilde{\nabla}$ denotes the Riemannian connection of $\mathbb{R}^{n+2}$, then

\begin{align*}
\alpha_p(X) = \overline{\nabla}_X \widetilde{v} = \widetilde{\nabla}_X \widetilde{v} - \langle \widetilde{\nabla}_X \widetilde{v}, p \rangle p, \quad X \in T_p M.
\end{align*}

A straightforward calculation shows that

\begin{align}
\widetilde{\nabla}_X \widetilde{v} = \left[ \frac{-\langle v, X \rangle (1 + \langle p, p_0 \rangle) + \langle v, p \rangle \langle X, p_0 \rangle}{(1 + \langle p, p_0 \rangle)^2} \right] &(p + p_0) - \left[ \frac{\langle v, p \rangle}{1 + \langle p, p_0 \rangle} \right] X. \label{v na direcao X}
\end{align}

\noindent Notice that $\langle X, \widetilde{v}(p) \rangle = \langle X, \eta(p) \rangle = 0$, since $X \in T_p M$. Using the expression of $\widetilde{v}$ in (\ref{v}), we have 

\begin{align*}
-\langle v, X \rangle (1 + \langle p, p_0 \rangle) + \langle v, p \rangle \langle X, p_0 \rangle = 0.
\end{align*}

\noindent Substituting this in (\ref{v na direcao X}), we obtain

\begin{align*}
\widetilde{\nabla}_X \widetilde{v} &= - \left[ \frac{\langle v, p \rangle}{1 + \langle p, p_0 \rangle} \right] X.
\end{align*}

\noindent Then, using formula (ii) for $\gamma(p)$,

\begin{align*}
\langle v, p \rangle &= \langle \gamma(p), p \rangle 
                      = - \left[ \frac{\langle \eta(p), p_0 \rangle}{1 + \langle p, p_0 \rangle} \right] \langle p + p_0, p \rangle + \langle \eta(p), p \rangle 
                      = - \langle \eta(p), p_0 \rangle.
\end{align*}

\noindent Hence, 

\begin{align}
\alpha_p(X) = \widetilde{\nabla}_X \widetilde{v}             
            = \left[ \frac{\langle \eta(p), p_0 \rangle}{1 + \langle p, p_0 \rangle} \right] X, \quad X \in T_p M.
\end{align}

\noindent since $\widetilde{\nabla}_X \widetilde{v}$ is already tangent to the sphere at $p$.
\end{proof}
 
Before we prove Theorem \ref{our rigidity}, we need the following lemma:
 
\begin{lem} \label{lemma} For a parameter $r \in (0,1)$, let
 
\begin{align*}
M_r = \mathbb{S}^1(r) \times \mathbb{S}^{n-1}(s) = \left\lbrace (x,y) \in \mathbb{R}^2 \times \mathbb{R}^n : \norm{x} = r, \norm{y} = s \right\rbrace \subset \mathbb{S}^{n+1},
\end{align*}

\noindent where $s = \sqrt{1 - r^2}$. If $R$ is the radius of the largest open geodesic ball of $\mathbb{S}^{n+1}$ which does not intersect $M_r$, then

\begin{align*}
\cos R = \min \{ r, s \}.
\end{align*}
\end{lem}

\begin{proof}
Recall that the distance between two points $p,q$ in the sphere $\mathbb{S}^{n+1}$ is given by $\arccos \langle p, q \rangle$, so that 

\begin{align*}
\cos R = \inf \left\lbrace \sup \{ \langle p, q \rangle : q \in M_r \} : p \in \mathbb{S}^{n+1} \right\rbrace.
\end{align*}

\noindent Writing $p = (x,y) \in \mathbb{R}^2 \times \mathbb{R}^n$, we have

\begin{align*}
\sup \{ \langle p, q \rangle : q \in M_r \} &= \sup \{ \langle x, u \rangle + \langle y, v \rangle : (u, v) \in M_r \} \\
&= r \norm{x} + s \norm{y}.
\end{align*}

\noindent Thus, 

\begin{align*}
\cos R = \inf \left\lbrace r \norm{x} + s \norm{y} : (x,y) \in \mathbb{S}^{n+1} \right\rbrace = \min \{ r, s \}.
\end{align*}
\end{proof}
 
\begin{proof}[Proof of Theorem \ref{our rigidity}]
Let $\eta : M \to \mathbb{R}^{n+2}$ be the unit normal vector field which gives rise to the orientation of $M$ and let $p_0$ be the center of a geodesic ball of radius $R$ 
containing $M$. Define a function $c : M \to \mathbb{R}$ by

\begin{align*}
c(p) = \frac{\langle \eta(p), p_0 \rangle}{1 + \langle p, p_0 \rangle}, \quad p \in M
\end{align*}

\noindent and a vector field $E \in \mathfrak{X}(\mathbb{S}^{n+1})$ by

\begin{align*}
E(p) = p_0 - \langle p, p_0 \rangle p, \quad p \in M.
\end{align*}

\noindent Notice that $\langle \eta(p), E(p) \rangle = \langle \eta(p), p_0 \rangle$ for $p$ in $M$. Then, using Cauchy-Schwarz inequality, we have the following estimate
for $c$:

\begin{align*}
|c(p)| \leq \frac{\norm{\eta(p)} \norm{E(p)}}{1 + \langle p, p_0 \rangle} = \frac{\sqrt{ 1 - \langle p, p_0 \rangle^2 }}{1 + \langle p, p_0 \rangle} = 
\sqrt{ \frac{ 1 - \langle p, p_0 \rangle }{1 + \langle p, p_0 \rangle} }, \quad \forall \, p \in M.
\end{align*}

\noindent Thus,

\begin{align*}
|c(p)| \leq \sqrt{ \frac{ 1 - \cos d(p, p_0) }{1 + \cos d(p, p_0)} } = \tan \left( \frac{d(p, p_0)}{2} \right) \leq \tan \left( \frac{R}{2} \right), \quad \forall \, p \in M.
\end{align*}

Let $p \in M$. Choosing an orthonormal basis of $T_p M$ that diagonalises the shape operator $A_p$, the matrix of $\Gamma_p^{-1} \circ D \gamma(p)$ with respect to this
basis is diagonal with entries $\lambda_i(p) + c(p) \neq 0$. Therefore, this map is an isomorphism for each $p \in M$, and so is $D\gamma(p)$. Since $M$ is compact, $\gamma$
is a covering map, and since $M$ is connected with $n \geq 2$, $\gamma$ is a diffeomorphism.

For the second part, let $\varepsilon \in (0, \sqrt{2} - 1)$. We will show that it is possible to choose $r \in I = \left( 0, \tfrac{1}{\sqrt{2}} \Big] \right.$ so that 
the principal curvatures of the hypersurface $M_r \subset\mathbb{S}^{n+1}$ from Lemma \ref{lemma} satisfy $(\ref{epsilon_inequality})$.

For any $r \in (0,1)$, the principal curvatures $\lambda_i$ of $M_r$ are constant, with

\begin{align*}
\lambda_1 = -\frac{\sqrt{1-r^2}}{r}
\end{align*}

\noindent and 

\begin{align*}
\lambda_2 =  \cdots = \lambda_n = \frac{r}{\sqrt{1-r^2}}.
\end{align*}

\noindent If $r \in I$ then $r \leq \sqrt{1-r^2}$ and, according to Lemma \ref{lemma}, $\cos R = r$. A simple calculation then shows that $(\ref{epsilon_inequality})$ holds
if and only if $r \in J_\varepsilon = \left( \frac{\varepsilon}{1 - \varepsilon}, \frac{1}{1 + \varepsilon}  \right)$. Since $\varepsilon \in (0,\sqrt{2} - 1)$, we have 
$J_\varepsilon \neq \emptyset$ and $I \cap J_\varepsilon \neq \emptyset$. Thus, any $r$ in this intersection is suitable for our purposes.
\end{proof}

In order to prove Theorem \ref{xia rigidity}, we start by introducing some ingredients and notations. Let $p_0$ be the north pole of $\mathbb{S}^{n+1}$ and let $\mathbb S_+^{n+1}$
be the open hemisphere centered at $p_0$. The Beltrami map $B : S_+^{n+1} \to \mathbb{R}^{n+1} \approx T_{p_0}\mathbb S^{n+1}$ is the diffeomorphism obtained by central projection.
Explicitly, it is given by

\begin{align*}
B(p) = \left( \frac{p_1}{p_{n+2}}, \dots, \frac{p_{n+1}}{p_{n+2}} \right), \quad p = (p_1, \dots p_{n+2}) \in S_+^{n+1}.
\end{align*} 

For $t > 0$, let $H_t : \mathbb{R}^{n+1} \to \mathbb{R}^{n+1}$ be the homothety $x \mapsto tx$. The map we are interested in is $C_t = B^{-1} \circ H_t \circ B$. It can be shown 
that

\begin{align*}
C_t(p) = \frac{m_t(p)}{\norm{m_t(p)}}, \quad p \in S_+^{n+1},
\end{align*}

\noindent where $m_t : S_+^{n+1} \to \mathbb{R}^{n+2} \setminus \{0\}$ is defined by

\begin{align}
m_t(p) = \left( p_1, \dots, p_{n+1}, \frac{p_{n+2}}{t} \right), \quad p = (p_1, \dots p_{n+2}) \in S_+^{n+1}. \label{m_t}
\end{align}
 
\noindent Some long but easy calculations yield

\begin{align*} 
DC_t(p) \cdot v = \frac{1}{\norm{m_t(p)}} \left\lbrace \left[ \frac{(t-1) \langle v, p_0 \rangle}{t^2 \norm{m_t(p)}^2} \right] \left[(t+1) \langle p, p_0 \rangle p - t p_0 \right] + v  \right\rbrace, 
\end{align*}

\noindent for $(p, v) \in T S_+^{n+1}$.

Let $M$ be an oriented hypersurface of $\mathbb{S}^{n+1}$ with unit normal vector field $\eta : M \to \mathbb{R}^{n+2}$. Recall that the second fundamental form of $M$ at
a point $p$ (in the direction of $\eta$) is the quadratic form $\mathrm{II}_p : T_p M \to \mathbb{R}$ induced by the shape operator $A_p$, that is,

\begin{align*}
\mathrm{II}_p(v) = \langle A_p(v), v \rangle, \quad v \in T_p M.
\end{align*}

\noindent Alternatively, if $\alpha : (-\varepsilon, \varepsilon) \to M$ is a curve with $\alpha(0) = p$ and $\alpha'(0) = v$, then

\begin{align*}
\mathrm{II}_p(v) = \langle \alpha''(0), \eta(p) \rangle,
\end{align*}

\noindent where the double prime indicates the usual second derivative, regarding $\alpha$ as a curve in $\mathbb{R}^{n+2}$.

\begin{proof} [Proof of Theorem \ref{xia rigidity}] After a rotation, we may suppose $M$ is contained in $\mathbb{S}_+^{n+1}$. By Theorem \ref{our rigidity} (with
$R = \tfrac{\pi}{2}$), $M$ would be diffeomorphic to $\mathbb{S}^n$ if all its principal curvatures were bigger that $1$ in absolute value. This is not necessarily true.
However, defining $M_t = C_t(M)$, we will show that if $t$ is sufficiently small, then this bound on the principal curvatures holds for $M_t$.

Let $\eta : M \to \mathbb{R}^{n+2}$ be the unit normal vector field (tangent do the sphere) that induces the orientation of $M$. One can directly check that the vector 
field $\eta_t : M_t \to \mathbb{R}^{n+2}$ given by

\begin{align}
\eta_t(C_t(p)) = \frac{\eta(p) + (t-1) \langle \eta(p), p_0 \rangle p_0}{ \sqrt{1 + (t^2 - 1) \langle \eta(p), p_0 \rangle^2}}, \quad p \in M \label{eta_t},
\end{align}

\noindent is normal to $M_t$; it has unit length because the denominator is the norm of the numerator.

We will establish a relationship between the second fundamental forms $\mathrm{II}$ and $\mathrm{II}^t$ of $M$ and $M_t$ with respect to the normals $\eta$ and $\eta_t$. 
Let $\alpha : (-\varepsilon, \varepsilon) \to M$ be a curve with $\alpha(0) = p$ and $\alpha'(0) = v$, with $\norm{v} = 1$. Consider $\beta = C_t \circ \alpha$ the 
corresponding curve in $M_t$, with $\beta(0) = q$ and $\beta'(0) = w$. 

Introducing the functions $y_t, z_t : M \to \mathbb{R}$ given by

\begin{align*}
y_t(p) &= \frac{(t+1) \langle p, p_0 \rangle }{t \norm{m_t(p)} }, \quad p \in M
\end{align*}

\noindent and 

\begin{align*}
z_t(p) &= \frac{1}{ \norm{m_t(p)} }, \quad p \in M,
\end{align*}

\noindent one has, after rearranging,

\begin{align*}
\beta'(s) = z_t(\alpha(s)) \left\lbrace \left[ \frac{(t-1) \langle \alpha'(s), p_0 \rangle}{t} \right] \left[ y_t(\alpha(s)) \beta(s) - p_0 \right] + \alpha'(s) \right\rbrace.
\end{align*}

\noindent Differentiating this expression and evaluating at $s = 0$, we obtain

\begin{align*}
\beta''(0) = \left( Dz_t(p) \cdot v \right) \norm{m_t(p)} w + z_t(p) \left\lbrace \left[ \frac{(t-1) \langle \alpha''(0), p_0 \rangle }{t} \right] \left[ y_t(p) q - p_0 \right] \right. &\\
 + \left. \left[ \frac{(t-1) \langle v, p_0 \rangle }{t} \right] \left[ \left( Dy_t(p) \cdot v \right)q + y_t(p) w \right] + \alpha''(0) \right\rbrace&.
\end{align*}

\noindent Since $\langle q, \eta_t(q) \rangle = \langle w, \eta_t(q) \rangle = 0$, we have

\begin{align*}
\langle \beta''(0), \eta_t(q) \rangle = z_t(p) \left[ \langle \alpha''(0), \eta_t(q) \rangle - \frac{(t-1) \langle \alpha''(0), p_0 \rangle \langle \eta_t(q), p_0 \rangle}{t} \right].
\end{align*}

\noindent Using expression (\ref{eta_t}) for $\eta_t$ we arrive at

\begin{align*}
\mathrm{II}_q^t(w) = \langle \beta''(0), \eta_t(q) \rangle = \frac{\mathrm{II}_p(v)}{\norm{m_t(p)} \left[1 + (t^2-1) \langle \eta(p), p_0 \rangle^2 \right]^{1/2}}.
\end{align*}

\noindent Furthermore, 

\begin{align*}
\norm{w}^2= \frac{1}{\norm{m_t(p)}^2} \left[ \frac{(1-t^2)(\langle p, p_0 \rangle^2 + \langle v, p_0 \rangle^2) + t^2}{t^2 \norm{m_t(p)}^2} \right].
\end{align*}

\noindent Thus, these two last equations and the value of $\norm{m_t(p)}$ obtainable from (\ref{m_t}) yield the desired relationship between $\mathrm{II}_p$ and $\mathrm{II}_q^t$:

\begin{align*}
\mathrm{II}_q^t \left( \frac{w}{\norm{w}} \right) = F_t(p,v) \mathrm{II}_p(v), 
\end{align*}

\noindent where 

\begin{align*}
F_t(p,v) = \frac{\left[ (1-t^2) \langle p, p_0 \rangle^2 + t^2 \right]^{3/2}}{t \left[ (1-t^2) (\langle p, p_0 \rangle^2 + \langle v, p_0 \rangle^2) + t^2 \right] 
\left[ 1 + (t^2-1) \langle \eta(p), p_0 \rangle^2 \right]^{1/2}}.
\end{align*}

Since $M$ is compact and contained in $\mathbb{S}_+^{n+1}$ we may choose $h, \varepsilon \in (0,1)$ such that $\langle x, p_0 \rangle^2 \geq h$  and $\langle \eta(x), p_0 \rangle^2 < 1 - \varepsilon^2$ for all $x \in M$. We have the following estimates if $0 < t < \tfrac{1}{\sqrt{2}}$:

\begin{gather*}
(1-t^2) \langle p, p_0 \rangle^2 + t^2 \geq \frac{h}{2} \\
(1-t^2) (\langle p, p_0 \rangle^2 + \langle v, p_0 \rangle^2) + t^2 \leq 3 \\
1 + (t^2 - 1) \langle \eta(p), p_0 \rangle^2 \leq 1.
\end{gather*}

\noindent This way, 

\begin{align*}
F_t(p,v) \geq \frac{K}{t}, \quad \forall \, p \in M, \; \forall \, v \in T_p M, \, \norm{v} = 1,
\end{align*}

\noindent for $K = \tfrac{h^{3/2}}{6 \sqrt{2}}$.

Let $\lambda_1 \leq \cdots \leq \lambda_n$ and $\mu_{1,t} \leq \cdots \leq \mu_{n,t}$ be the principal curvatures of $M$ and $M_t$, respectively. The variational principle
for eigenvalues gives

\begin{align*}
\lambda_j(p) = \min \left\lbrace \max \left\lbrace \mathrm{II}_p(v) : v \in V, \, \norm{v} = 1 \right\rbrace : V \subseteq T_p M, \, \dim V = j \right\rbrace
\end{align*}

\noindent and

\begin{align*}
\mu_{j,t}(C_t(p)) = \min \left\lbrace \max \left\lbrace F_t(p,v) \mathrm{II}_p(v) : v \in V, \, \norm{v} = 1 \right\rbrace : V \subseteq T_p M, \, \dim V = j \right\rbrace.
\end{align*}

\noindent Notice that $M$ must contain an elliptic point, that is, a point where all principal curvatures have the same sign, which we assume to be positive. The connectedness 
of $M$ and the fact that its Gauss-Kronecker curvature is nowhere zero implies all principal curvatures are everywhere positive. So, for every point $p \in M$ and subspace 
$V$ of $T_p M$, we have

\begin{align*}
\max \left\lbrace F_t(p,v) \mathrm{II}_p(v) : v \in V, \, \norm{v} = 1 \right\rbrace \geq F_t(p, v(V)) \, \mathrm{II}_p(v(V)) \geq \frac{K}{t} \, \mathrm{II}_p(v(V)),
\end{align*} 

\noindent where $v(V) \in V$ satisfies $\norm{v(V)} = 1$ and 

\begin{align*}
\mathrm{II}_p(v(V)) = \max \left\lbrace \mathrm{II}_p(v) : v \in V, \, \norm{v} = 1 \right\rbrace > 0.
\end{align*}

\noindent Hence, 

\begin{align*}
\mu_{j,t}(C_t(p)) &\geq \min \left\lbrace \frac{K}{t} \mathrm{II}_p(v(V)) : V \subseteq T_p M, \, \dim V = j \right\rbrace \\
&=\frac{K}{t} \min \left\lbrace \max \left\lbrace \mathrm{II}_p(v) : v \in V, \, \norm{v} = 1 \right\rbrace : V \subseteq T_p M, \, \dim V = j \right\rbrace \\
&=\frac{K}{t} \lambda_j(p).
\end{align*}

\noindent Setting  

\begin{align*}
\lambda = \min \left\lbrace \lambda_j(p) : p \in M, \, j \in \{1, \dots, n\} \right\rbrace > 0,
\end{align*}

\noindent we have

\begin{align*}
\mu_{j,t}(C_t(p)) \geq \frac{K}{t} \lambda
\end{align*} 

\noindent for every $p \in M$ and $0< t < \tfrac{1}{\sqrt{2}}$. Thus, provided that $t$ is sufficiently small, all the principal curvatures of $M_t$ are bigger than 
$1$ in absolute value, as we wanted.
\end{proof}

\begin{remark}
We observe that the same constructions done in the sphere can be done in the hyperbolic space using the Lorentzian model. In particular, one can prove
using the same technique that a compact hypersurface of the hyperbolic space having everywhere nonzero Gauss-Kronecker curvature is diffeomorphic to a
sphere. However, in the hyperbolic space, if  the Gauss-Kronecker curvature is nowhere zero then necessarily the principal curvatures of the hypersurface
have the same sign, and then the result follows from the Proposition of \cite{ripollpams}.
\end{remark}

\end{document}